\documentclass[12pt,a4paper]{article}
\usepackage[ngerman, english]{babel}
\usepackage{enumerate}
\usepackage{amsmath,amssymb, amsthm, amsfonts}
\usepackage[colorlinks]{hyperref}
\usepackage{comment}

\theoremstyle{plain}
\newtheorem {lemma}{Lemma}

\newtheorem {theorem}[lemma]{Theorem}
\newtheorem {corollary}[lemma]{Corollary}

\theoremstyle{definition}
\newtheorem{definition}[lemma]{Definition}
\newtheorem{remark}[lemma]{Remark}

\newtheorem {question}[lemma]{Question}

\newcommand{\N}{\mathbb{N}}

\newcommand{\+}{\overset{.}{+}}
\newcommand{\minus}{\overset{.}{-}}
\newcommand{\h}{\mathfrak{H}}
\newcommand{\tr}{\operatorname{tr}}
\newcommand{\hb}{\operatorname{hb}}
\newcommand{\GL}{\operatorname{GL}}

\newcommand{\U}{\operatorname{U}}
\newcommand{\EU}{\operatorname{EU}}
\newcommand{\NU}{\operatorname{NU}}
\newcommand{\CU}{\operatorname{CU}}
\newcommand{\Ort}{\operatorname{O}}
\newcommand{\Sp}{\operatorname{Sp}}
\newcommand{\Normaliser}{\operatorname{Normaliser}}
\newcommand{\Mat}{\operatorname{M}}
\newcommand{\cn}{\operatorname{cn}}
\newcommand{\scn}{\operatorname{scn}}
\newcommand{\id}{\operatorname{id}}
\newcommand{\diag}{\operatorname{diag}}
\newcommand{\short}{\operatorname{short}}
\newcommand{\extra}{\operatorname{extra}}

\title{Elementary covering numbers in odd-dimensional unitary groups}
\author{Raimund Preusser}
\date{}
\AtEndDocument{\bigskip{\footnotesize%
  \textsc{Chebyshev Laboratory, St. Petersburg State University, Russia} \par  
  \textit{E-mail address:} \texttt{raimund.preusser@gmx.de} \par
}}

\begin{document}
\maketitle
\begin{abstract}
\noindent
Let $(K,\Delta)$ be a Hermitian form field and $n\geq 3$. We prove that if $\sigma\in \U_{2n+1}(K,\Delta)$ is a unitary matrix of level $(K,\Delta)$, then any short root transvection $T_{ij}(x)$ is a product of $4$ elementary unitary conjugates of $\sigma$ and $\sigma^{-1}$. Moreover, the bound $4$ is sharp. We also show that any extra short root transvection $T_i(x,y)$ is a product of $12$ elementary unitary conjugates of $\sigma$ and $\sigma^{-1}$. If the level of $\sigma$ is $(0,K\times 0)$, then any $(0,K\times 0)$-elementary extra short root transvection $T_i(x,0)$ is a product of $2$ elementary unitary conjugates of $\sigma$ and $\sigma^{-1}$.
\end{abstract}
\let\thefootnote\relax
\let\thefootnote\relax
\let\thefootnote\relax\footnotetext{The work is supported by the Russian Science Foundation grant 19-71-30002.}

\section{Introduction}
The investigation of products of conjugacy classes in different types of groups is a popular topic in group theory during the last 30-40 years. Many papers were devoted to
this theme, for example \cite{c1,c2,c3,c4,c4_2,gordeev-saxl, gordeev-saxl_2,c5,c6,c7,c8,c9,c10,c11}. A lot of these works are concerned with the computation of covering numbers. 

Let $G$ be a group and and $S\subseteq G$ a subset. For any subset $X\subseteq G$ define
$\cn_{X}(S)$ as the least positiv integer $m$ such that $S\subseteq X^m=\{x_1\dots x_m\mid x_1,\dots,x_m\in X\}$. If no such $m$ exists, then $\cn_{X}(S):=\infty$. We call $\cn_X(S)$ the {\it covering number of $S$ with respect to $X$}. For any set $\mathcal{X}$ of subsets of $G$ define $\cn_{\mathcal{X}}(S)$ as the supremum of all covering numbers $\cn_X(S)$ where $X\in \mathcal{X}$. We call $\cn_\mathcal{X}(S)$ the {\it covering number of $S$ with respect to $\mathcal{X}$}. Note that if $\mathcal{X}$ is the set of all conjugacy classes in $G$ that are not contained in a proper normal subgroup, then $\cn_\mathcal{X}(G)$ is the usual covering number $\cn(G)$ as defined in \cite{gordeev-saxl}. We call $\scn_X(S):=\cn_{X\cup X^{-1}}(S)$ the {\it symmetric covering number of $S$ with respect to $X$} and $\scn_\mathcal{X}(S):=\sup\{\scn_X(S)\mid X\in \mathcal{X}\}$ the {\it symmetric covering number of $S$ with respect to $\mathcal{X}$}.

The hyperbolic unitary groups $\U_{2n}(R,\Lambda)$ were defined by A. Bak in 1969 \cite{bak}. They embrace the classical Chevalley groups of type $C_m$ and $D_m$, namely the even-dimensional symplectic and orthogonal groups $\Sp_{2n}(R)$ and $\Ort_{2n}(R)$. In 2018, A. Bak and the author  defined odd-dimensional unitary groups $\U_{2n+1}(R,\Delta)$ \cite{bak-preusser}. These groups generalise the even-dimensional unitary groups $\U_{2n}(R,\Lambda)$ and embrace all classical Chevalley groups. The groups $\U_{2n+1}(R,\Delta)$ are in turn embraced by V. Petrov's odd unitary groups, which were introduced in \cite{petrov}.

Let $(K,\Delta)$ be a Hermitian form field and $n\geq 3$. Denote the odd unitary group $\U_{2n+1}(K,\Delta)$ by $G$ and its elementary subgroup $\EU_{2n+1}(K,\Delta)$ by $E$. It follows from the Sandwich Classification Theorem \ref{SCT} that if $H$ is a subgroup of $G$ normalised by $E$, then there is a unique odd form ideal $(I,\Omega)$ of $(K,\Delta)$ such that
\[\EU_{2n+1}((K,\Delta),(I,\Omega))\subseteq H \subseteq \CU_{2n+1}((K,\Delta),(I,\Omega))\]
where $\EU_{2n+1}((K,\Delta),(I,\Omega))$ denotes the relative elementary subgroup of level $(I,\Omega)$ and $\CU_{2n+1}((K,\Delta),(I,\Omega))$ the full congruence subgroup of level $(I,\Omega)$. The odd form ideal $(I,\Omega)$ is called the \emph{level} of $H$. The {\it level} of a conjugacy class $C$ in $G$ is the level of the subgroup of $G$ generated by $C$.

Let $\mathcal{C}$ denote the set of all conjugacy classes of level $(K,\Delta)$, $S_{\short}$ the set of all nontrivial short root transvections and $S_{\extra}$ the set of all nontrivial extra short root transvections. In this paper we prove that $\scn_{\mathcal{C}}(S_{\short})\leq 4$ and $\scn_{\mathcal{C}}(S_{\extra})\leq 12$. Moreover, we show that the bound $\scn_{\mathcal{C}}(S_{\short})\leq 4$ is sharp, i.e. there is no better bound valid for all Hermitian form fields $(K,\Delta)$ and $n\geq 3$. 

If the Hermitian form $B$ is degenerate and $K\times 0\subseteq \Delta$, then there is a second nonzero odd form ideal, namely $(0,K\times 0)$. Let $\mathcal{D}$ denote the set of all conjugacy classes of level $(0,K\times 0)$ and $T$ the set of all nontrivial $(0,K\times 0)$-elementary extra short root transvections. We prove that $\scn_{\mathcal{D}}(T)=1$ if $K=\mathbb{F}_2$ and $(0,1)\in \Delta$, and $\scn_{\mathcal{D}}(T)=2$ otherwise.  

The rest of the paper is organised as follows. In Section 2 we recall some standard notation which is used throughout the paper. In Section 3, we recall the definitions of the groups $\U_{2n+1}(R,\Delta)$ and some important subgroups. In Section 4, we prove our main results, namely Theorems \ref{thmm1}, \ref{thmm2}, \ref{thmm3} and \ref{thmm4}. The results of Section 4 are still valid if one replaces all occurrences of ``conjugacy class'' by ``$E$-class'' or more generally by ``$H$-class'' where $E\leq H\leq G$ is a fixed intermediate group. 

\section{Notation}
$\mathbb{N}$ denotes the set of all positive integers. If $G$ is a group and $g,h\in G$, we let $g^h:=h^{-1}gh$, $^hg:=hgh^{-1}$ and $[g,h]:=ghg^{-1}h^{-1}$. By a ring we mean an associative ring with $1$ such that $1\neq 0$. By an ideal we mean a two-sided ideal. If $m,n\in \N$ and $R$ is a ring, then the set of all $m\times n$ matrices over $R$ is denoted by $\Mat_{m\times n}(R)$. Instead of $\Mat_{n\times n}(R)$ we may write $\Mat_n(R)$. If $\sigma\in M_{m\times n}(R)$, we denote the transpose of $\sigma$ by $\sigma^t$, the entry of $\sigma$ at position $(i,j)$ by $\sigma_{ij}$, the $i$-th row of $\sigma$ by $\sigma_{i*}$ and the $j$-th column of $\sigma$ by $\sigma_{*j}$. The group of all invertible matrices in $M_{n}(R)$ is denoted by $\GL_n(R)$ and the identity element of $\GL_n(R)$ by $e$ or $e_{n\times n}$. If $\sigma\in \GL_n(R)$, then the entry of $\sigma^{-1}$ at position $(i,j)$ is denoted by $\sigma'_{ij}$, the $i$-th row of $\sigma^{-1}$ by $\sigma'_{i*}$ and the $j$-th column of $\sigma^{-1}$ by $\sigma'_{*j}$.  Furthermore, we denote by $^n\!R$ the set of all row vectors of length $n$ with entries in $R$ and by $R^n$ the set of all column vectors of length $n$ with entries in $R$. We consider $^n\!R$ as left $R$-module and $R^n$ as right $R$-module.

\section{Odd-dimensional unitary groups}
We describe Hermitian form rings $(R,\Delta)$ and odd form ideals $(I,\Omega)$ first, then the odd-dimensional unitary group $\U_{2n+1}(R,\Delta)$ and its elementary subgroup $\EU_{2n+1}(R,\Delta)$ over a Hermitian form ring $(R,\Delta)$. For an odd form ideal $(I,\Omega)$, we recall the definitions of the following subgroups of $\U_{2n+1}(R,\Delta)$: the preelementary subgroup $\EU_{2n+1}(I, \Omega)$ of level $(I,\Omega)$, the elementary subgroup $\EU_{2n+1}((R,\Delta),(I,\Omega))$ of level $(I,\Omega)$, the principal congruence subgroup $\U_{2n+1}((R,\Delta),(I,\Omega))$ of level $(I,\Omega)$, the normalised principal congruence subgroup $\NU_{2n+1}((R,\Delta),(I,\Omega))$ of level $(I,\Omega)$, and the full congruence subgroup $\CU_{2n+1}((R,\Delta),(I,\Omega))$ of level $(I,\Omega)$.

\subsection{Hermitian form rings and odd form ideals}\label{sec 3.1}
First we recall the definitions of a ring with involution with symmetry and a Hermitian ring.
\begin{definition}
Let $R$ be a ring and 
\begin{align*}
\bar{}:R&\rightarrow R\\
x&\mapsto \bar{x}
\end{align*}
an anti-isomorphism of $R$ (i.e. $\bar{}~$ is bijective, $\overline{x+y}=\bar x+\bar y$, $\overline{xy}=\bar y\bar x$ for any $x,y\in R$ and $\bar 1=1$). Furthermore, let $\lambda\in R$ such that $\bar{\bar x}=\lambda x\bar\lambda$ for any $x\in R$. Then $\lambda$ is called a {\it symmetry} for $~\bar{}~$, the pair $(~\bar{}~,\lambda)$ an {\it involution with symmetry} and the triple $(R,~\bar{}~,\lambda)$ a {\it ring with involution with symmetry}. A subset $A\subseteq R$ is called {\it involution invariant} iff $\bar x\in A$ for any $x\in A$. A {\it Hermitian ring} is a quadruple $(R,~\bar{}~,\lambda,\mu )$ where $(R,~\bar{}~,\lambda)$ is a ring with involution with symmetry and $\mu \in R$ is a ring element such that $\mu =\bar\mu \lambda$ .
\end{definition}
\begin{remark}\label{25}
Let $(R,~\bar{}~,\lambda,\mu )$ be a Hermitian ring.
\begin{enumerate}[(a)]
\item It is easy to show that $\bar\lambda=\lambda^{-1}$.
\item The map
\begin{align*}
\b{}:R&\rightarrow R\\
x& \mapsto \b{x}:=\bar\lambda \bar x\lambda
\end{align*}
is the inverse map of $~\bar{}~$. One checks easily that $(R,~\b{}~,\b{$\lambda$},\b{$\mu $})$ is a Hermitian ring.
\end{enumerate}
\end{remark}

Next we recall the definition of an $R^{\bullet}$-module.
\begin{definition}
If $R$ is a ring, let $R^\bullet$ denote the underlying set of the ring equipped with the  
multiplication of the ring, but not the addition of the ring. A {\it (right) $R^{\bullet}$-module} is a not 
necessarily abelian group $(G,\+)$ equipped with a map
\begin{align*}
\circ: G\times R^{\bullet}&\rightarrow G\\
(a,x) &\mapsto a\circ x
\end{align*}
such that the following holds:
\begin{enumerate}[(i)]
\item $a\circ 0=0$ for any $a\in G$,
\item $a\circ 1=a$ for any $a\in G$,
\item $(a\circ x)\circ y=a\circ (xy)$ for any $a\in G$ and $x,y\in R$ and
\item $(a\+ b)\circ x=(a\circ x)\+(b\circ x)$ for any $a,b\in G$ and $x\in R$.
\end{enumerate}
Let $G$ and $G'$ be $R^{\bullet}$-modules. A group homomorphism $f:G\rightarrow G'$ satisfying $f(a\circ x)=f(a)\circ x$ for any $a\in G$ and $x\in R$ is called a {\it  homomorphism of $R^{\bullet}$-modules}. A subgroup $H$ of $G$ which is $\circ$-stable (i.e. $a\circ x\in H$ for any $a\in H$ and $x\in R$) is called an {\it $R^{\bullet}$-submodule}. Moreover, if $A\subseteq G$ and $B\subseteq R$, we denote by $A\circ B$ the subgroup of $G$ generated by $\{a\circ b\mid a\in A,b\in B\}$. We treat $\circ$ as an operator with higher priority than $\+$.
\end{definition}

An important example of an $R^{\bullet}$-module is the Heisenberg group, which we define next. 

\begin{definition}\label{27}
Let $(R,~\bar{}~,\lambda,\mu )$ be a Hermitian ring. Define the map.
\begin{align*}
\+: (R\times R)\times (R\times R) &\rightarrow R\times R\\
((x_1,y_1),(x_2,y_2))&\mapsto (x_1,y_1)\+ (x_2,y_2):=(x_1+x_2,y_1+y_2-\bar x_1\mu  x_2).
\end{align*}
Then $(R\times R,\+)$ is a group, which we call the {\it Heisenberg group} and denote by $\h$. Equipped with the map
\begin{align*}
\circ:(R\times R)\times R^{\bullet}&\rightarrow R\times R\\
((x,y),a)&\mapsto (x,y)\circ a:=(xa,\bar aya)
\end{align*}
$\h$ becomes an $R^{\bullet}$-module. 
\end{definition}
\begin{remark}
We denote the inverse of an element $(x,y)\in \h$ by $\minus(x,y)$. One checks easily that $\minus(x,y)=(-x,-y-\bar x\mu  x)$ for any $(x,y)\in \h$.
\end{remark}

In order to define the odd-dimensional unitary groups we need the notion of a Hermitian form ring.
\begin{definition}
Let $(R,~\bar{}~,\lambda,\mu )$ be a Hermitian ring. Let $(R,+)$ have the $R^{\bullet}$-module structure defined by $x\circ a = \bar{a}xa$. Define the {\it trace map}
\begin{align*}
\tr:\h&\rightarrow R\\
(x,y)&\mapsto \bar x\mu  x+y+\bar y\lambda.
\end{align*}
One checks easily that $\tr$ is a homomorphism of $R^{\bullet}$-modules. Set \[\Delta_{\min}:=\{(0,x-\overline{x}\lambda)\mid x\in R\}\] and \[\Delta_{\max}:=\ker(\tr).\] An $R^{\bullet}$-submodule $\Delta$ of $\h$ lying between $\Delta_{\min}$ and $\Delta_{\max}$ is called an {\it odd form parameter} for $(R,~\bar{}~,\lambda,\mu )$. Since $\Delta_{\min}$ and $\Delta_{\max}$ are $R^{\bullet}$-submodules of $\h$, they are respectively the smallest and largest odd form parameters. A pair $((R,~\bar{}~,\lambda,\mu ),\Delta)$ is called a {\it Hermitian form ring}. We shall usually abbreviate it by $(R,\Delta)$. 
\end{definition}
 
Next we define an odd form ideal of a Hermitian form ring. 
\begin{definition}
Let $(R,\Delta)$ be a Hermitian form ring and $I$ an involution invariant ideal of $R$. Set $J(\Delta):=\{y\in R\mid\exists z\in R:(y,z)\in \Delta\}$ and $\tilde I:=\{x\in R\mid\overline{J(\Delta)}\mu  x\subseteq I\}$. Obviously $\tilde I$ and $J(\Delta)$ are right ideals of $R$ and $I\subseteq \tilde I$. Moreover, set \[\Omega^I_{\min}:=\{(0,x-\bar x\lambda)\mid x\in I\}\+ \Delta\circ I\] and \[\Omega^I_{\max}:=\Delta\cap (\tilde I\times  I).\]
An $R^{\bullet}$-submodule $\Omega$ of $\h$ lying between $\Omega^I_{\min}$ and $\Omega^I_{\max}$ is called a {\it relative odd form parameter of level $I$}. Since $\Omega^I_{\min}$ and $\Omega^I_{\max}$ are $R^{\bullet}$-submodules of $\h$, they are respectively the smallest and the largest relative odd form parameters of level $I$. If $\Omega$ is a relative odd form parameter of level $I$, then $(I,\Omega)$ is called an {\it odd form ideal} of $(R,\Delta)$.
\end{definition}

\subsection{The odd-dimensional unitary group}
Let $(R,\Delta)$ be a Hermitian form ring and $n\in \mathbb{N}$. Set $M:=R^{2n+1}$. We use the following indexing for the elements of the standard basis of $M$: $(e_1,\dots,e_n,e_0,e_{-n},\dots,e_{-1})$.
That means that $e_i$ is the column whose $i$-th coordinate is one and all the other coordinates are zero if $1 \leq i\leq n$, the column whose $(n+1)$-th coordinate is one and all the other coordinates are zero if $i=0$, and the column whose $(2n+2+i)$-th coordinate is one and all the other coordinates are zero if $-n\leq i \leq -1$. If $u\in M$, then we call $(u_1,\dots,u_n,u_{-n},\dots,$ $u_{-1})^t\in R^{2n}$ the {\it hyperbolic part} of $u$ and denote it by $u_{\hb}$. We set $u^*:=\bar u^t$ and $u_{\hb}^*:=\bar u_{\hb}^t$. Moreover, we define the maps
\begin{align*}
B:M\times M&\rightarrow R\\
(u,v)&\mapsto u^*\begin{pmatrix} 0& 0 & \pi\\0&\mu &0\\ \pi\lambda &0 &0 \end{pmatrix}v=\sum\limits_{i=1}^{n}\bar u_i v_{-i}+\bar u_0\mu  v_0+\sum\limits_{i=-n}^{-1}\bar u_{i}\lambda v_{-i}
\end{align*}
and 
\begin{align*}
Q:M&\rightarrow \h\\
u&\mapsto (Q_1(u),Q_2(u)):=(u_0,u_{\hb}^*\begin{pmatrix} 0&\pi\\0&0 \end{pmatrix}u_{\hb})=(u_0,\sum\limits_{i=1}^{n}\bar u_i u_{-i})			
\end{align*}
where $\pi\in M_n(R)$ denotes the matrix with ones on the skew diagonal and zeros elsewhere.

\begin{lemma}[{\cite[Lemma 12]{bak-preusser}}]
~\\
\vspace{-0.6cm}
\begin{enumerate}[(i)]
\item $B$ is a \textnormal{$\lambda$-Hermitian form}, i.e. $B$ is biadditive, $B(ux,vy)=\bar x B(u,v) y~\forall u,v\in M,x,y\in R$ and $B(u,v)=\overline{B(v,u)}\lambda~\forall u,v\in M$.
\item $Q(ux)=Q(u)\circ x~\forall u\in M, x\in R$, $Q(u+v)\equiv Q(u)\+ Q(v)\+(0,B(u,v))\bmod \Delta_{\min}~\forall u,v\in M$ and $\tr(Q(u))=B(u,u)~\forall u\in M$.
\end{enumerate}
\end{lemma}

\begin{definition}
The group $\U_{2n+1}(R,\Delta):=$
\begin{align*}
\{\sigma\in \GL_{2n+1}(R)\mid B(\sigma u,\sigma v)=B(u,v)\land Q(\sigma u)\equiv Q(u)\bmod \Delta~\forall u,v\in M\}
\end{align*}
is called the {\it odd-dimensional unitary group}.
\end{definition}


\begin{remark}\label{34}
The groups $\U_{2n+1}(R,\Delta)$ include as special cases the even-dimensional unitary groups $\U_{2n}(R,\Lambda)$ and all classical Chevalley groups. On the other hand, the groups $\U_{2n+1}(R,\Delta)$ are embraced by Petrov's odd unitary groups $\U_{2l}(R,\mathfrak{L})$. For details see \cite[Remark 14(c) and Example 15]{bak-preusser}. 
\end{remark}

\begin{definition}
We define the sets $\Theta_+:=\{1,\dots,n\}$, $\Theta_-:=\{-n,\dots,-1\}$, $\Theta:=\Theta_+\cup\Theta_-\cup\{0\}$ and $\Theta_{\hb}:=\Theta\setminus \{0\}$. Moreover, we define the map 
\begin{align*}\epsilon:\Theta_{\hb} &\rightarrow\{\pm 1\}\\
i&\mapsto\begin{cases} 1, &\mbox{if } i\in\Theta_+, \\ 
-1, & \mbox{if } i\in\Theta_-. \end{cases}
\end{align*}
\end{definition}


\begin{lemma}[{\cite[Lemma 17]{bak-preusser}}]\label{36}
Let $\sigma\in \GL_{2n+1}(R)$. Then $\sigma\in \U_{2n+1}(R,\Delta)$ iff Conditions (i) and (ii) below are satisfied. 
\begin{enumerate}[(i)]
\item \begin{align*}
       \sigma'_{ij}&=\lambda^{-(\epsilon(i)+1)/2}\bar\sigma_{-j,-i}\lambda^{(\epsilon(j)+1)/2}~\forall i,j\in\Theta_{\hb},\\
       \mu \sigma'_{0j}&=\bar\sigma_{-j,0}\lambda^{(\epsilon(j)+1)/2}~\forall j\in\Theta_{\hb},\\
       \sigma'_{i0}&=\lambda^{-(\epsilon(i)+1)/2}\bar\sigma_{0,-i}\mu ~\forall i\in\Theta_{\hb} \text { and}\\
       \mu \sigma'_{00}&=\bar\sigma_{00}\mu .
      \end{align*}
\item \begin{align*}
Q(\sigma_{*j})\equiv (\delta_{0j},0)\bmod \Delta ~\forall j\in \Theta.
\end{align*} 
\end{enumerate}
\end{lemma}

\begin{lemma}\label{lemcol}
Let $\sigma\in \U_{2n+1}(R,\Delta)$. If $\sigma_{*j}=e_kx$ for some $j,k\in\Theta_{\hb}
$ and invertible $x\in R$, then $\sigma_{-k,*}=(e_{-j}\hat x)^t$ where $\hat x=\lambda^{(\epsilon(k)-1)/2}\bar x^{-1}\lambda^{(1-\epsilon(j))/2}$.
\end{lemma}
\begin{proof}
Since $e=\sigma^{-1}\sigma$, we have
\[\delta_{ij}=(\sigma^{-1}\sigma)_{ij}=\sigma'_{i*}\sigma_{*j}=\sigma'_{i*}e_kx=\sigma'_{ik}x\]
for any $i\in\Theta$. It follows from the previous lemma that $\sigma_{-k,*}=(e_{-j}\hat x)^t$.
\end{proof}

\subsection{The polarity map}
\begin{definition}
The map
\begin{align*}
\widetilde{}~: M&\longrightarrow M^*\\
u&\longmapsto \begin{pmatrix} \bar u_{-1}\lambda&\dots&\bar u_{-n}\lambda&\bar u_0\mu&\bar u_{n}&\dots&\bar u_1\end{pmatrix}
\end{align*}
where $M^*={}^{2n+1}\!R$ is called the {\it polarity map}. Clearly $~\widetilde{}~$ is {\it involutary linear}, i.e. $\widetilde{u+v}=\tilde u+\tilde v$ and $\widetilde{ux}=\bar x\tilde u$ for any $u,v\in M$ and $x\in R$.
\end{definition}
\begin{lemma}[{\cite[Lemma 16]{preusser_subnormal}}]\label{38}
If $\sigma\in \U_{2n+1}(R,\Delta)$ and $u\in M$, then $\widetilde{\sigma u}=\tilde u\sigma^{-1}$.
\end{lemma}

\subsection{The elementary subgroup}
We introduce the following notation. Let $(R,~\b{}~,\b{$\lambda$},\b{$\mu $})$ be the Hermitian ring defined in Remark \ref{25}(b) and $\h^{-1}$ the corresponding Heisenberg group. Note that the underlying set of both $\h$ and $\h^{-1}$ is $R\times R$. We denote the group operation (resp. scalar multiplication) of $\h$ by $\+_1$ (resp. $\circ_1$) and the group operation (resp. scalar multiplication) of $\h^{-1}$ by $\+_{-1}$ (resp. $\circ_{-1}$). Furthermore, we set $\Delta^1:=\Delta$ and $\Delta^{-1}:=\{(x,y)\in R\times R\mid(x,\bar y)\in \Delta\}$. One checks easily that $((R,~\b{}~,\b{$\lambda$},\b{$\mu $}),\Delta^{-1})$ is a Hermitian form ring. Analogously, if $(I,\Omega)$ is an odd form ideal of $(R,\Delta)$, we set $\Omega^1:=\Omega$ and $\Omega^{-1}:=\{(x,y)\in R\times R\mid(x,\bar y)\in \Omega\}$. One checks easily that $(I,\Omega^{-1})$ is an odd form ideal of $(R,\Delta^{-1})$. 
 
If $i,j\in \Theta$, let $e^{ij}$ denote the matrix in $M_{2n+1}(R)$ with $1$ in the $(i,j)$-th position and $0$ in all other positions.
\begin{definition}
If $i,j\in \Theta_{\hb}$, $i\neq\pm j$ and $x\in R$, the element  \[T_{ij}(x):=e+xe^{ij}-\lambda^{(\epsilon(j)-1)/2}\bar x\lambda^{(1-\epsilon(i))/2}e^{-j,-i}\] of $\U_{2n+1}(R,\Delta)$ is called an {\it (elementary) short root transvection}. 
If $i\in \Theta_{\hb}$ and $(x,y)\in \Delta^{-\epsilon(i)}$, the element \[T_{i}(x,y):=e+xe^{0,-i}-\lambda^{-(1+\epsilon(i))/2}\bar x\mu e^{i0}+ye^{i,-i}\] of $\U_{2n+1}(R,\Delta)$ is called an {\it (elementary) extra short root transvection}. The extra short root transvections of the kind \[T_{i}(0,y)=e+ye^{i,-i}\] are called {\it (elementary) long root transvections}. If an element of $\U_{2n+1}(R,\Delta)$ is a short or extra short root transvection, then it is called an {\it elementary transvection}. The subgroup of $\U_{2n+1}(R,\Delta)$ generated by all elementary transvections is called the {\it elementary subgroup} and is denoted by $\EU_{2n+1}(R,\Delta)$. 
\end{definition}

\begin{lemma}[{\cite[Lemma 20]{bak-preusser}}]\label{39}
The following relations hold for the elementary transvections.
\begin{align*}
&T_{ij}(x)=T_{-j,-i}(-\lambda^{(\epsilon(j)-1)/2}\bar x\lambda^{(1-\epsilon(i))/2}), \tag{S1}\\
&T_{ij}(x)T_{ij}(y)=T_{ij}(x+y), \tag{S2}\\
&[T_{ij}(x),T_{kl}(y)]=e \text{ if } k\neq j,-i \text{ and } l\neq i,-j, \tag{S3}\\
&[T_{ij}(x),T_{jk}(y)]=T_{ik}(xy) \text{ if } i\neq\pm k, \tag{S4}\\
&[T_{ij}(x),T_{j,-i}(y)]=T_{i}(0,xy-\lambda^{(-1-\epsilon(i))/2}\bar y\bar x\lambda^{(1-\epsilon(i))/2}), \tag{S5}\\
&T_{i}(x_1,y_1)T_{i}(x_2,y_2)=T_{i}((x_1,y_1)\+_{-\epsilon(i)}(x_2,y_2)), \tag{E1}\\
&[T_{i}(x_1,y_1),T_{j}(x_2,y_2)]=T_{i,-j}(-\lambda^{-(1+\epsilon(i))/2}\bar x_1\mu x_2) \text{ if } i\neq\pm j, \tag{E2}\\
&[T_{i}(x_1,y_1),T_{i}(x_2,y_2)]=T_{i}(0,-\lambda^{-(1+\epsilon(i))/2}(\bar x_1\mu x_2-\bar x_2\mu x_1)), \tag{E3}\\
&[T_{ij}(x),T_{k}(y,z)]=e \text{ if } k\neq j,-i \text{ and} \tag{SE1}\\
&[T_{ij}(x),T_{j}(y,z)]=T_{j,-i}(z\lambda^{(\epsilon(j)-1)/2}\bar x\lambda^{(1-\epsilon(i))/2})\cdot\\
&\hspace{3.4cm}\cdot T_{i}(y\lambda^{(\epsilon(j)-1)/2}\bar x\lambda^{(1-\epsilon(i))/2},xz\lambda^{(\epsilon(j)-1)/2}\bar x\lambda^{(1-\epsilon(i))/2}).\tag{SE2}\\
\end{align*}
\end{lemma}

\begin{definition}\label{defdiag}
Let $x\in R$ be invertible and $i,j\in\Theta_{\hb}$ such that $i\neq\pm j$. Define
\begin{align*}
D_{ij}(x):=&e+(x-1)e^{ii}+(x^{-1}-1)e^{jj}+(\lambda^{(\epsilon(i)-1)/2}\bar x^{-1}\lambda^{-(\epsilon(i)-1)/2}-1)e^{-i,-i}\\
&+(\lambda^{(\epsilon(j)-1)/2}\bar x\lambda^{-(\epsilon(j)-1)/2}-1)e^{-j,-j}\\
=&T_{ij}(x-1)T_{ji}(1)T_{ij}(x^{-1}-1)T_{ji}(-x)\in \EU_{2n+1}(R,\Delta).                                                                                                                                                                                                                                                                                                                                                                                                                                                                                             
\end{align*}
Clearly $(D_{ij}(x))^{-1}=D_{ij}(x^{-1})$. 
\end{definition}

\begin{definition}\label{defperm}
Let $i,j\in\Theta_{\hb}$ such that $i\neq\pm j$. Define
\begin{align*}
P_{ij}:=&e-e^{ii}-e^{jj}-e^{-i,-i}-e^{-j,-j}+e^{ij}-e^{ji}+\lambda^{(\epsilon(i)-\epsilon(j))/2}e^{-i,-j}-\lambda^{(\epsilon(j)-\epsilon(i))/2}e^{-j,-i}\\
=&T_{ij}(1)T_{ji}(-1)T_{ij}(1)\in \EU_{2n+1}(R,\Delta).                                                                                                                                                                                                                                                                                                                                                                                                                                                                                             
\end{align*}
Clearly $(P_{ij})^{-1}=P_{ji}$. 
\end{definition}

The two lemmas below are easy to check.
\begin{lemma}\label{lemdiag}
Let $i,j,k\in\Theta_{\hb}$ such that $i\neq \pm j$ and $k\neq \pm i,\pm j$. Let $a\in R$ be invertible, $x\in R$ and $(y,z)\in\Delta^{-\epsilon(i)}$. Then 
\begin{enumerate}[(i)]
\item $^{D_{ik}(a)
}T_{ij}(x)=T_{ij}(ax)$,
\item $^{D_{kj}(a)}T_{ij}(x)=T_{ij}(xa)$ and
\item $^{D_{-i,k}(a^{-1})}T_{i}(y,z)=T_{i}(ya,\lambda^{-(\epsilon(i)+1)/2}\bar a\lambda^{(\epsilon(i)+1)/2}za)$.
\end{enumerate}
\end{lemma} 

\begin{lemma}[{\cite[Lemma 23]{bak-preusser}}]\label{lemperm}
Let $i,j,k\in\Theta_{\hb}$ such that $i\neq \pm j$ and $k\neq \pm i,\pm j$. Let $x\in R$ and $(y,z)\in\Delta^{-\epsilon(i)}$. Then 
\begin{enumerate}[(i)]
\item $^{P_{ki}}T_{ij}(x)=T_{kj}(x)$,
\item $^{P_{kj}}T_{ij}(x)=T_{ik}(x)$ and
\item $^{P_{-k,-i}}T_{i}(y,z)=T_{k}(y,\lambda^{(\epsilon(i)-\epsilon(k))/2}z)$.
\end{enumerate}
\end{lemma}

\subsection{Relative subgroups}
In this subsection $(I,\Omega)$ denotes an odd form ideal of $(R,\Delta)$.
\begin{definition}
A short root transvection $T_{ij}(x)$ is called {\it $(I,\Omega)$-elementary} if $x\in I$. An extra short root transvection $T_{i}(x,y)$ is called {\it $(I,\Omega)$-elementary} if $(x,y)\in \Omega^{-\epsilon(i)}$. 
The subgroup $\EU_{2n+1}(I,\Omega)$ of $\EU_{2n+1}(R,\Delta)$ generated by the $(I,\Omega)$-elementary trans\-vections is called the {\it preelementary subgroup of level $(I,\Omega)$}. Its normal closure $\EU_{2n+1}((R,\Delta),(I,\Omega))$ in $\EU_{2n+1}(R,\Delta)$ is called the {\it elementary subgroup of level $(I,\Omega)$}.
\end{definition}

If $\sigma\in M_{2n+1}(R)$, we call the matrix $(\sigma_{ij})_{i,j\in\Theta_{\hb}}\in M_{2n}(R)$ the {\it hyperbolic part} of $\sigma$ and denote it by $\sigma_{\hb}$. Furthermore, we define the submodule $M(R,\Delta):=\{u\in M\mid u_0\in J(\Delta)\}$ of $M$. 
\begin{definition}
The subgroup $\U_{2n+1}((R,\Delta),(I,\Omega)):=$
\[\{\sigma\in \U_{2n+1}(R,\Delta)\mid\sigma_{\hb}\equiv e_{\hb}\bmod  I\text{ and }Q(\sigma u)\equiv Q(u)\bmod \Omega~\forall u\in M(R,\Delta)\}\]
of $\U_{2n+1}(R,\Delta)$ is called {\it the principal congruence subgroup of level $(I,\Omega)$}.
\end{definition}

\begin{lemma}[{\cite[Lemma 28]{bak-preusser}}]\label{46}
Let $\sigma\in \U_{2n+1}(R,\Delta)$. Then $\sigma\in \U_{2n+1}((R,\Delta),(I,\Omega))$ iff Conditions (i) and (ii) below are satisfied. 
\begin{enumerate}[(i)]
\item $\sigma_{\hb}\equiv e_{\hb}\bmod  I$.
\item $Q(\sigma_{*j})\in\Omega~\forall j\in \Theta_{\hb}$ and $(Q(\sigma_{*0})\minus (1,0))\circ a\in\Omega~\forall a\in J(\Delta)$.
\end{enumerate}
\end{lemma}


\begin{definition}
The subgroup $\NU_{2n+1}((R,\Delta),(I,\Omega)):=$
\[\Normaliser_{\U_{2n+1}(R,\Delta)}(\U_{2n+1}((R,\Delta),(I,\Omega)))\]
of $\U_{2n+1}(R,\Delta)$ is called the {\it normalised principal congruence subgroup of level $(I,\Omega)$}.
\end{definition}


\begin{definition}
The subgroup $\CU_{2n+1}((R,\Delta),(I,\Omega)):=$
\[ \{\sigma\in \NU_{2n+1}((R,\Delta),(I,\Omega))\mid [\sigma,\EU_{2n+1}(R,\Delta)]\leq \U_{2n+1}((R,\Delta),(I,\Omega))\}\]
of $\U_{2n+1}(R,\Delta)$ is called the {\it full congruence subgroup of level $(I,\Omega)$}.
\end{definition}


\subsection{The standard commutator formulas and the sandwich classification theorem}
We call the ring $R$ {\it quasifinite}, if it is a direct limit of subrings $R_i~(i\in \Phi)$ which are almost commutative (i.e. finitely generated as modules over their centers), involution invariant and contain $\lambda$ and $\mu$. 

\begin{theorem}[{\cite[Theorem 39]{bak-preusser}}]\label{SCF}
Suppose that $R$ is quasifinite and $n\geq 3$. Then $\EU_{2n+1}((R,\Delta),(I,\Omega))$ is a normal subgroup of $\NU_{2n+1}((R,\Delta),(I,\Omega))$ and the standard commutator formulas
\begin{align*}
&[\CU_{2n+1}((R,\Delta),(I,\Omega)),\EU_{2n+1}(R,\Delta)]\\
=&[\EU_{2n+1}((R,\Delta),(I,\Omega)),\EU_{2n+1}(R,\Delta)]\\
=&\EU_{2n+1}((R,\Delta),(I,\Omega))
\end{align*}
hold. In particular from the absolute case $(I,\Omega)=(R,\Delta)$, it follows that $\EU_{2n+1}(R,\Delta)$ is perfect and normal in $\U_{2n+1}(R,\Delta)$. 
\end{theorem}

\begin{theorem}[{\cite[Theorem 80]{bak-preusser}}]\label{SCT}
Suppose that $R$ is quasifinite and $n\geq 3$. Let $H$ be a subgroup of $\U_{2n+1}(R,\Delta)$. Then $H$ is normalised by $\EU_{2n+1}(R,\Delta)$ if and only if there is an odd form ideal $(I,\Omega)$ of $(R,\Delta)$ such that 
\[\EU_{2n+1}((R,\Delta),(I,\Omega))\subseteq H \subseteq \CU_{2n+1}((R,\Delta),(I,\Omega)).\]
Moreover, $(I,\Omega)$ is uniquely determined.
\end{theorem}

Recall that if $H$ is a subgroup of $\U_{2n+1}(R,\Delta)$ normalised by $\EU_{2n+1}(R,\Delta)$, then the uniquely determined odd form ideal $(I,\Omega)$ in Theorem \ref{SCT} is called the level of $H$.

\section{Elementary covering numbers in $\U_{2n+1}(K,\Delta)$}
In this section $n\geq 3$ denotes an integer and $(K,\Delta)$ a Hermitian form field (i.e. $(K,\Delta)$ is a Hermitian form ring and $K$ a field). We denote the odd unitary group $\U_{2n+1}(K,\Delta)$ by $G$ and its elementary subgroup $\EU_{2n+1}(K,\Delta)$ by $E$.

If $\mu\neq 0$ or $K\times 0\not\subseteq \Delta$, then there are only two odd form ideals in $(K,\Delta)$, namely $(0,0)$ and $(K,\Delta)$. If $\mu=0$ and $K\times 0\subseteq \Delta$, then there is a third odd form ideal, namely $(0,K\times 0)$. It follows from \cite[Remark 26]{bak-preusser} that $\NU_{2n+1}((K,\Delta),(I,\Omega))=\U_{2n+1}(K,\Delta)$ for any odd form ideal $(I,\Omega)$. Hence $\EU_{2n+1}((K,\Delta),(I,\Omega))$, $\U_{2n+1}((K,\Delta),(I,\Omega))$ and $\CU_{2n+1}((K,\Delta),(I,\Omega))$ are normal subgroups of $\U_{2n+1}(K,\Delta)$. 

Recall that the level of a conjugacy class $C$ in $G$ is the level of the subgroup generated by $C$. In Subsection 4.1 we investigate covering numbers with respect to conjugacy classes of level $(K,\Delta)$. In Subsection 4.2 we investigate covering numbers with respect to conjugacy classes of level $(0,K\times 0)$. 

\begin{lemma}\label{lemlev}
Let $C$ be a conjugacy class in $G$ and $\sigma$ an element of $C$. Set
\[Y:=\{\sigma_{ij},\sigma_{ii}-\sigma_{jj}, \sigma_{i0}J(\Delta),\overline{J(\Delta)}\mu\sigma_{0j},\overline{J(\Delta)}\mu(\sigma_{00}-\sigma_{jj})J(\Delta)\mid i,j\in \Theta_{\hb},i\neq j\}\]
and 
\[Z:=\{Q(\sigma_{*j}),(Q(\sigma_{*0})\minus(1,0))\circ y\+(y,z)\minus (y,z)\circ \sigma_{ii}\mid i,j\in \Theta_{\hb}, (y,z)\in \Delta
\}.\]
Let $I$ be the involution invariant ideal generated by $Y$ and set $\Omega:=\Omega_{\min}^I\+
Z\circ K$. Then $(I,\Omega)$ is the level of $C$.
\end{lemma}
\begin{proof}
If $(J,\Sigma)$ and $(J',\Sigma')$ are odd form ideals, then we call $(J,\Sigma)$ {\it smaller} than $(J',\Sigma')$ if $J\subseteq J'$ and $\Sigma\subseteq \Sigma'$.
It order to prove the assertion of the lemma, it suffices to show that $(I,\Omega)$ is the smallest odd form ideal such that $\langle C\rangle\subseteq \CU_{2n+1}((R,\Delta),(I,\Omega))$. But that follows from \cite[Lemma 31]{preusser_subnormal}.
\end{proof}

\subsection{Elementary covering numbers with respect to conjugacy classes of level $(K,\Delta)$}

We denote by $\mathcal{C}$ the set of all conjugacy classes of level $(K,\Delta)$, by $S_{\short}$ the set of all nontrivial short root transvections and by $S_{\extra}$ the set of all nontrivial extra short root transvections. We will prove that $\scn_{\mathcal{C}}(S_{\short})\leq 4$ and $\scn_{\mathcal{C}}(S_{\extra})\leq 12$. In order to do that we need three lemmas.

\begin{lemma}\label{lem1}
Any two elements of $S_{\short}$ are conjugated. 
\end{lemma}
\begin{proof}
The lemma follows from Lemmas \ref{lemdiag} and \ref{lemperm}.
\end{proof}

In the following lemma we drop the assumption that $n\geq 3$.
\begin{lemma}\label{lem2}
Let $n\geq 1$ and $\sigma\in \U_{2n+1}(K,\Delta)$. Then either $({}^{\tau}\!\sigma)_{*1}=e_{-1}x$ for some $\tau\in \EU_{2n+1}(K,\Delta)$ and $x\in K$ or $({}^{\tau}\!\sigma)_{*1}=e_{2}x$ for some $\tau\in \EU_{2n+1}(K,\Delta)$ and $x\in K$ or $\sigma_{*1}=e_{1}x+e_0y$ for some $x,y\in K$.
\end{lemma}
\begin{proof}
First suppose that $\sigma_{-1,1}\neq 0$. Then $({}^{\tau}\!\sigma)_{*1}=e_{-1}x$ for some $x\in K$ where $\tau=(\prod_{i\neq 0,\pm 1} T_{i,-1}(*))T_1(*)$. Now suppose that $\sigma_{-1,1}=0$ and $\sigma_{j1}\neq 0$ for some $j\neq 0,\pm 1$. We may assume that $j=2$ (conjugate $\sigma$ by a product of $P_{kl}$'s). Clearly $({}^{\tau}\!\sigma)_{*1}=e_{2}x$ for some $x\in K$ where $\tau=(\prod_{i\neq 0,-1,\pm 2} T_{i2}(*))T_{-2}(*)$. The assertion of the lemma follows.
\end{proof}

\begin{lemma}\label{lem3}
Let $\sigma\in G$ such that $\sigma_{*1}=e_{2}x$ for some $x\in K$. Then there is a $\tau\in E$ such that $({}^{\tau}\!\sigma)_{*1}=e_{2}x$ and $({}^{\tau}\!\sigma)_{i,-2}=0$ for some $i\in \{\pm 3\}$.
\end{lemma}
\begin{proof}
We may asume that $\sigma_{\pm 3,-2}\neq 0$ (otherwise we can choose $\tau=e$). First suppose that $\sigma_{-1,-2}\neq 0$. Then the assertion of the lemma holds with $\tau=T_{-3,-1}(-\sigma_{-3,-2}$ $(\sigma_{-1,-2})^{-1})$. Suppose now that $\sigma_{-1,-2}= 0$ (note that we also have $\sigma_{-2,-2}=0$ by Lemma \ref{lemcol}). Then the assertion of the lemma holds with 
\[\tau=(\prod_{i\neq 0,\pm 1,\pm 2, \pm 3} T_{i,-3}(*))T_{-3}(*).\]
\end{proof}

\begin{theorem}\label{thmm1}
Let $\mathcal{C}$ denote the set of all conjugacy classes of level $(K,\Delta)$ and $S_{\short}$ the set of all nontrivial short root transvections. Then $\scn_{\mathcal{C}}(S_{\short})\leq 4$.
\end{theorem}
\begin{proof}
Let $C\in\mathcal{C}$ and $\sigma\in C$. In order to prove the theorem it suffices to show that $\scn_{C}(S_{\short})\leq 4$. 
\begin{enumerate}[{\bf Case} 1]
\item Suppose that $\sigma_{ij}\neq 0$ for some $i,j\in\Theta_{\hb}, i\neq j$. Then there is a product $\tau$ of $P_{kl}$'s such that ${}(^{\tau}\!\sigma)_{t1}\neq 0$ for some $t\neq 0,1$. The proof of Lemma \ref{lem2} shows that there is a $\rho\in E$ such that either $(^{\rho\tau}\!\sigma)_{*1}= e_{-1}x$ or $(^{\rho\tau}\!\sigma)_{*1}= e_{2}x$ for some $x\in K$. Set $\zeta:={}^{\rho\tau}\!\sigma$.
\begin{enumerate}[{\bf Subcase}{ 1.}1] 
\item Suppose that $\zeta_{*1}=e_{-1}x$. It follows from Lemma $\ref{lemcol}$ that
\[\zeta=\begin{pmatrix}0&0&\hat x \\ 0&A&v \\ x&u&z\end{pmatrix}  \]
for some $A\in \Mat_{2n-1}(K)$, $u\in \Mat_{1\times (2n-1)}(K)$, $v\in \Mat_{(2n-1)\times 1}(K)$ and $\hat x, z\in K$. Clearly $A\in \U_{2n-1}(K,\Delta)$ by Lemma \ref{36}. By Lemma \ref{lem2} we may assume that $\zeta_{-3,2}$ (the penultimate entry of the first column of $A$) equals zero. One checks easily that $[T_{31}(1),[T_{1,-2}(1),\zeta]]=T_{3,-2}(1)$. It follows that $T_{3,-2}(1)\in CC^{-1}CC^{-1}$. Thus $S_{\short}\subseteq CC^{-1}CC^{-1}$ by Lemma \ref{lem1}.
\item Suppose that $\zeta_{*1}=e_{2}x$. By Lemma \ref{lem3} we may assume that $\zeta_{i,-2}=0$ for some $i\in \{\pm 3\}$. By Lemma \ref{lemcol} we have $\zeta_{-2,*}=(e_{-1}\hat x)^t$ for some $\hat x\in K$. One checks easily that $[T_{2i}(-1),[T_{12}(1),$ $\zeta]]=T_{1i}(1)$. It follows that $T_{1i}(1)\in CC^{-1}CC^{-1}$. Thus $S_{\short}\subseteq CC^{-1}CC^{-1}$ by Lemma \ref{lem1}.
\end{enumerate}
\item Suppose that $\sigma_{ij}= 0$ for any $i,j\in\Theta_{\hb}, i\neq j$ and  $\sigma_{kk}\neq\sigma_{ll}$ for some $k,l\in\Theta_{\hb}, k\neq l$. Clearly $({}^{T_{kl}(1)}\!\sigma)_{kl}=\sigma_{ll}-\sigma_{kk}\neq 0$ and hence one can apply Case 1 to ${}^{T_{kl}(1)}\!\sigma$.
\item Suppose that $\sigma_{ij}, \sigma_{jj}-\sigma_{ii}= 0$ for any $i,j\in\Theta_{\hb}, i\neq j$. Then $\sigma$ has the form 
\[\sigma=\left(\begin{array}{ccc|c|ccc}y&&&*&&&\\&\ddots&&\vdots&&&\\&&y&*&&&\\\hline *&\dots&*&*&*&\dots&* \\\hline&&&*&y&&\\&&&\vdots&&\ddots&\\&&&*&&&y\end{array}\right)\]
for some $y\in K$. Since the level of $C$ is $(K,\Delta)$, it follows from Lemma \ref{lemlev} that $\mu\neq 0$ (note that if $\mu=0$, then $\sigma_{i0}=0$ for any $i\in\Theta_{\hb}$ by Lemma \ref{36}).
\begin{enumerate}[{\bf Subcase}{ 3.}1] 
\item Suppose that $\sigma_{0j}\neq 0$ for some $j\in \Theta_{\hb}$. By Lemma \ref{36} we have $(\sigma_{0j},0)=Q(\sigma_{*j})\in \Delta$. Choose an $i\in\Theta_{\hb}$ such that $i\neq \pm j$. Then $({}^{T_i(\sigma_{0j},0)}\!\sigma)_{ij}\neq 0$ and hence one can apply Case 1 to ${}^{T_i(\sigma_{0j},0)}\!\sigma$.
\item Suppose that $\sigma_{0j}=0$ for any $j\in \Theta_{\hb}$. It follows from Lemma \ref{lemcol} that $\sigma$ is a diagonal matrix. Since the level of $C$ equals $(K,\Delta)$, it follows from Lemma \ref{lemlev} that $J(\Delta)=K$ and $\sigma_{00}\neq \sigma_{11}$. Since $J(\Delta)=K$ we can choose a $y\in K$ such that $(1,y)\in\Delta$. Clearly $({}^{T_{-1}(1,y)}\!\sigma)_{01}=\sigma_{11}-\sigma_{00}\neq 0$ and hence one can apply Case 1 or Subcase 3.1 to ${}^{T_{-1}(1,y)}\!\sigma$.
\end{enumerate}
\end{enumerate}
\end{proof}

The corollary below follows from Relation (SE2) in Lemma \ref{39}.
\begin{corollary}\label{corm1}
Let $\mathcal{C}$ denote the set of all conjugacy classes of level $(K,\Delta)$ and $S_{\extra}$ the set of all nontrivial extra short root transvections. Then $\scn_{\mathcal{C}}(S_{\extra})\leq 12$.
\end{corollary}

\begin{theorem}\label{thmm2}
Suppose that $~\bar{}~=\id$, $\lambda=-1$, $\mu=1$ and $\Delta=\Delta_{\max}=0\times K$ (hence $G$ is isomorphic to the symplectic group $\Sp_{2n}(K)$). Moreover, suppose that $K$ has characteristic $2$. Then $\scn_{\mathcal{C}}(S_{\short})=3$ or $\scn_{\mathcal{C}}(S_{\short})=4$.
\end{theorem}
\begin{proof}
In view of Theorem \ref{thmm1} it suffices to find an $C\in\mathcal{C}$ such that $\scn_{C}(S_{\short})\geq 3$. Let $C$ be the conjugacy class of $T_{1}(0,1)$. Note that $C=C^{-1}$ since $K$ has characteristic $2$. Assume that $T_{12}(1)\in C$. Then there is a $\sigma\in H$ such that ${}^{\sigma}\!T_{1}(0,1)=T_{12}(1)$. Let $u$ be the first column of $\sigma$. It follows from Lemma \ref{38} that ${}^{\sigma}\!T_{1}(0,1)=e+u\tilde u$. Since by assumption ${}^{\sigma}\!T_{1}(0,1)=T_{12}(1)$, we obtain $u_1\neq 0$. But then 
\[0=(T_{12}(1))_{1,-1}=(e+u\tilde u)_{1,-1}=u_1^2\neq 0\]
which is absurd.\\
Assume now that $T_{12}(1)\in CC$. Then there are $\sigma,\tau\in H$ such that
\begin{align*}
&({}^{\sigma}T_{1}(0,1))({}^{\tau}T_{1}(0,1))=T_{12}(1)\\
\Leftrightarrow ~&{}^{\sigma}T_{1}(0,1)=T_{12}(1)({}^{\tau}T_{1}(0,1)).
\end{align*}
Let $u$ and $v$ be the first columns of $\sigma$ and $\tau$, respectively. It follows from Lemma \ref{38} that
\begin{equation}
e+u\tilde u=T_{12}(1)(e+v\tilde v).
\end{equation}
Let $i\in \Theta\setminus\{1,-2\}$. It follows from Equation (1) that 
\begin{equation}
u_i\tilde u=v_i\tilde v.
\end{equation}
Clearly either $u_i,v_i\neq 0$ or $u_i,v_i=0$. Assume $u_i,v_i\neq 0$. Then $\tilde v=v_i^{-1}u_i\tilde u$ which implies that $v=uk$ for some nonzero $k\in K$. Choose a $j\in\Theta$ such that $(\tilde u)_j\neq 0$. It follows from Equation (2) that $u_i(\tilde u)_j=v_i(\tilde v)_j=k^2u_i(\tilde u)_j$ whence $k^2=1$. Hence $v\tilde v=uk^2\tilde u=u\tilde u$ which leads to a contradiction (consider the first two rows of the matrices in Equation (1)). Hence we have shown that $u_i,v_i=0$ for any $i\in \Theta\setminus\{1,-2\}$. By considering the entries of the matrices in Equation (1) at positions $(1,2)$, $(1,-1)$ and $(-2,2)$, we obtain $u_1u_{-2}=v_1v_{-2}+1$, $u_1^{2}=v_1^{2}$ and $u_{-2}^{2}=v_{-2}^{2}$. It follows that $(u_1u_{-2})^{2}=(v_1v_{-2})^2+1=(u_1u_{-2})^{2}+1$ which is absurd. \\
We have shown that neither $S_{\short}\subseteq C$ nor $S_{\short}\subseteq CC$. It follows that $\scn_{C}(S_{\short})\geq 3$. Thus $\scn_{\mathcal{C}}(S_{\short})=3$ or $\scn_{\mathcal{C}}(S_{\short})=4$ by Theorem \ref{thmm1}.
\end{proof}

\begin{theorem}\label{thmm3}
Suppose that $~\bar{}~=\id$, $\lambda=-1$, $\mu=0$ and $\Delta=\Delta_{\max}=K\times K$ (hence $G$ is Proctor's odd symplectic group $\Sp_{2n+1}(K)$, see \cite[Example 15(4)]{bak-preusser}). Moreover, suppose that $K$ has characteristic $2$ and contains an element of order $\geq 4$. Then $\scn_{\mathcal{C}}(S_{\short})=4$.
\end{theorem}
\begin{proof}
In view of Theorem \ref{thmm1} it suffices to find an $C\in\mathcal{C}$ such that $\scn_{C}(S_{\short})\geq 4$. Choose an $x\in K$ of order $\geq 4$ and set $\alpha:=\diag(1,\dots,1,x,1,\dots,1)$ $\in G$ where $x$ is at position $(0,0)$. Let $C$ be the conjugacy class of 
\[\beta:=\alpha T_{1}(0,1)=e+e^{1,-1}+(x-1)e^{00}.\]
Assume that $T_{12}(1)\in C^{i_1}\dots C^{i_m}$ for some $m\in\{1,2,3\}$ and $i_1,\dots,i_m\in\{\pm 1\}$. Since $\det(\beta)=x$ has order $\geq 4$, it follows that $m=2$ and $p_1=-p_2$.  We only consider the case $p_1=1$ and $p_2=-1$ and leave the case $p_1=-1$ and $p_2=1$ to the reader. So assume that $T_{12}(1)\in CC^{-1}$. Then there are $\sigma,\tau\in H$ such that
\begin{align*}
{}^{\sigma}\!\beta({}^{\tau}\!\beta^{-1})=T_{12}(1)~\Leftrightarrow ~{}^{\sigma}\!\beta=T_{12}(1)({}^{\tau}\!\beta).
\end{align*}
Let $u$ and $v$ be the first columns of $\sigma$ and $\tau$, respectively. It follows from Lemmas \ref{36} and \ref{38} that
\begin{equation}
e+u_{\hb}\tilde u_{\hb}+e_0w=T_{12}(1)(e+v_{\hb}\tilde v_{\hb}+e_0w')
\end{equation}
for some $w,w'\in \Mat_{1\times (2n+1)}(K)$.
Let $i\in \Theta_{hb}\setminus\{1,-2\}$. It follows from Equation (3) that 
\begin{equation}
u_i\tilde u_{\hb}=v_i\tilde v_{\hb}.
\end{equation}
Clearly either $u_i,v_i\neq 0$ or $u_i,v_i=0$. Assume $u_i,v_i\neq 0$. Then $\tilde v_{\hb}=v_i^{-1}u_i\tilde u_{\hb}$ which implies that $v_{\hb}=u_{\hb}k$ for some nonzero $k\in K$. Choose a $j\in\Theta_{\hb}$ such that $(\tilde u)_j\neq 0$. It follows from Equation (4) that $u_i(\tilde u)_j=v_i(\tilde v)_j=u_ik^2(\tilde u)_j$ whence $k^2=1$. Hence $v_{\hb}\tilde v_{\hb}=u_{\hb}\tilde u_{\hb}$ which leads to a contradiction (consider the first two rows of the matrices in Equation (3)). Hence we have shown that $u_i,v_i=0$ for any $i\in \Theta_{hb}\setminus\{1,-2\}$. By considering the entries of the matrices in Equation (3) at positions $(1,2)$, $(1,-1)$ and $(-2,2)$, we obtain $u_1u_{-2}=v_1v_{-2}+1$, $u_1^{2}=v_1^{2}$ and $u_{-2}^{2}=v_{-2}^{2}$. It follows that $(u_1u_{-2})^{2}=(v_1v_{-2})^2+1=(u_1u_{-2})^{2}+1$ which is absurd. \\
We have shown that there is no $m\in\{1,2,3\}$ and $i_1,\dots,i_m\in\{\pm 1\}$ such that $S_{\short}\subseteq C^{i_1}\dots C^{i_m}$. It follows that $\scn_{C}(S_{\short})\geq 4$. Thus $\scn_{\mathcal{C}}(S_{\short})=4$ by Theorem \ref{thmm1}.
\end{proof}

\subsection{Elementary covering numbers with respect to conjugacy classes of level $(0,K\times 0)$}

In this subsection we assume that $\mu=0$ and $K\times 0\subseteq \Delta$. We denote by $\mathcal{D}$ the set of all conjugacy classes of level $(0,K\times 0)$ and by $T$ the set of all nontrivial $(0,K\times 0)$-elementary extra short root transvections. We will determine $\scn_{\mathcal{D}}(T)$.

\begin{lemma}\label{lem2.1}
Any two elements of $T$ are conjugated.
\end{lemma}
\begin{proof}
The lemma follows from Lemmas \ref{lemdiag} and \ref{lemperm}.
\end{proof}

\begin{lemma}\label{lem2.2}
Let $\mathcal{D}$ denote the set of all conjugacy classes of level $(0,K\times 0)$, and $S$ the set of all nontrivial $(0,K\times 0)$-elementary extra short root transvections. Then $\scn_{\mathcal{D}}(T)\leq 2$.
\end{lemma}
\begin{proof}
Let $D\in\mathcal{D}$ and $\sigma\in D$. In order to prove the theorem it suffices to show that $\scn_{D}(T)\leq 2$. Since the level of $D$ equals $(0,K\times 0)$, there is an $x\in K$ and $u,v\in \Mat_{1\times n}(K)$ such that 
\[\sigma=\begin{pmatrix}e_{n\times n}&0&0\\u&x&v\\0&0&e_{n\times n}\end{pmatrix}.\]
\begin{enumerate}[{\bf Case} 1]
\item Suppose that $\sigma_{0i}\neq 0$ for some $i\in\Theta_{\hb}$. We may assume that $\sigma_{0j}=0$ for some $j\in\Theta_{\hb}\setminus\{\pm i\}$ (conjugate $\sigma$ by $T_{ij}(-\sigma_{0i}^{-1}\sigma_{0j})$). One checks easily that $[\sigma,T_{i,-j}(1)]=
T_j(\sigma_{0i},0)$. It follows that $T_j(\sigma_{0i},0)\in DD^{-1}$. Thus $S\subseteq DD^{-1}$ by Lemma \ref{lem2.1}.
\item Suppose that $\sigma_{0i}=0$ for any $i\in\Theta_{\hb}$. Then $x\neq 1$ since the level of $D$ equals $(0,K\times 0)$. One checks easily that $[\sigma,T_{1}(1,0)]=
T_1(x-1,0)$. It follows that $T_1(x-1,0)\in DD^{-1}$. Thus $S\subseteq DD^{-1}$ by Lemma \ref{lem2.1}.
\end{enumerate}
\end{proof}

\newpage

\begin{theorem}\label{thmm4}
$\scn_{\mathcal{D}}(T)=1$ if $K=\mathbb{F}_2$ and $(0,1)\in \Delta$, and $\scn_{\mathcal{D}}(T)=2$ otherwise. 
\end{theorem}
\begin{proof}
\begin{enumerate}[{\bf Case} 1]
\item Suppose that $K=\mathbb{F}_2$ and $(0,1)\in\Delta$. Let $D\in\mathcal{D}$ and $\sigma\in D$. Since the level of $D$ equals $(0,K\times 0)$, there are $u,v\in \Mat_{1\times n}(K)$ such that 
\[\sigma=\begin{pmatrix}e_{n\times n}&0&0\\u&1&v\\0&0&e_{n\times n}\end{pmatrix}.\]
Moreover, $\sigma_{0i}=1$ for some $i\in\Theta_{\hb}$. One checks easily that $\sigma^\tau=T_{-i}(1,0)$ where $\tau=(\prod\limits_{j\neq \pm i}T_{ij}(*))T_i(0,*)$. It follows that $T_{-i}(1,0)\in D$. Thus $S\subseteq D$ by Lemma \ref{lem2.1}.
\item Suppose that $K=\mathbb{F}_2$ and $(0,1)\not\in\Delta$. Set $\sigma:=T_1(1,0)T_{-1}(1,0)\in G$ and let $D$ be the conjugacy class of $\sigma$. Then the level of $D$ equals $(0,K\times 0)$. Assume that $\scn_{D}(T)=1$. Then there is a $\tau\in H$ such that ${}^{\tau}\sigma=T_{-1}(1,0)$. Since $\sigma=e+e_{01}+e_{0,-1}=e+e_0(e_1^t+e_{-1}^t)$, we obtain 
\begin{align*}
{}^{\tau}\sigma&=T_{-1}(1,0)\\
\Leftrightarrow~e+e_0(e_1^t+e_{-1}^t)\tau^{-1}&=e+e_{0}e_1^t\\
\Leftrightarrow\hspace{1.46cm}e_0(e_1^t+e_{-1}^t)&=e_{0}e_1^t\tau\\
\Leftrightarrow\hspace{1.46cm}e_0(e_1^t+e_{-1}^t)&=e_{0}\tau_{1*}\\
\Leftrightarrow\hspace{2.12cm}e_1^t+e_{-1}^t&=\tau_{1*}.
\end{align*}
It follows from Lemma \ref{36} that $\tau'_{*,-1}=e_1+e_0x+e_{-1}$ for some $x\in \mathbb{F}_2$. Hence $(x,1)=Q(\tau'_{*,-1})\in \Delta$ (also by Lemma \ref{36}). Since by assumption $(0,1)\not\in\Delta$, we obtain $(1,1)\in \Delta$. Since $(1,0)\in\Delta$ it follows that $(0,1)=(1,1)\+(1,0)\in\Delta$ which contradicts the assumption that $(0,1)\not\in\Delta$. Hence $\scn_{D}(T)\geq 2$. It follows from Lemma \ref{lem2.2} that $\scn_{\mathcal{D}}(T)=2$.
\item Suppose that $K\neq \mathbb{F}_2$. Choose an $x\in K\setminus\{0,1\}$. Let $D$ be the conjugacy class of $\diag(1,\dots,1,x,1,\dots,1)\in G$ where $x$ is in position $(0,0)$. Then the level of $D$ equals $(0,K\times 0)$. Since $\det(\diag(1,\dots,1,x,1,\dots,1))=x\neq 1$, we have $\scn_{D}(T)\geq 2$. It follows from Lemma \ref{lem2.2} that $\scn_{\mathcal{D}}(T)=2$.
\end{enumerate}
\end{proof}

\subsection{Some open questions}

As in Subsection 4.1 we denote by $\mathcal{C}$ the set of all conjugacy classes of level $(K,\Delta)$, by $S_{\short}$ the set of all nontrivial short root transvections and by $S_{\extra}$ the set of all nontrivial extra short root transvections. By Theorem \ref{thmm1} we have $\scn_{\mathcal{C}}(S_{\short})\leq 4$. By Theorem \ref{thmm3}, $4$ is the optimal uniform bound for $\scn_{\mathcal{C}}(S_{\short})$ (valid for all Hermitian form fields $(K,\Delta)$ and $n\geq 3$). One can ask Questions \ref{Q1} and \ref{Q2} below.

\begin{question}\label{Q1}
Can the bound $\scn_{\mathcal{C}}(S_{\short})\leq 4$ be improved if one restricts to Hermitian form fields $(K,\Delta)$ where $\mu\neq 0$ (i.e. the Hermitian form $B$ is nondegenerate) or $2$ is invertible?
\end{question}

\begin{question}\label{Q2}
What is the optimal bound for $\scn_{\mathcal{C}}(S_{\short})$ for the classical Chevalley groups $\Sp_{2n}(K)$, $\Ort_{2n}(K)$ and $\Ort_{2n+1}(K)$, respectively?
\end{question}

By Corollary \ref{corm1} we have $\scn_{\mathcal{C}}(S_{\extra})\leq 12$. But it could be the case that this bound is not optimal.

\begin{question}\label{Q3}
What is the optimal bound for $\scn_{\mathcal{C}}(S_{\extra})$?
\end{question}

\end{document}